\documentclass{article}

\usepackage{comment}

\usepackage{extarrows}

\usepackage[usenames,dvipsnames]{pstricks}
\usepackage{epsfig}
\usepackage{pst-grad} % For gradients
\usepackage{pst-plot} % For axes
\usepackage[space]{grffile} % For spaces in paths
\usepackage{etoolbox} % For spaces in paths
\makeatletter % For spaces in paths
\patchcmd\Gread@eps{\@inputcheck#1 }{\@inputcheck"#1"\relax}{}{}
\makeatother

\usepackage{fullpage}
\usepackage{amsmath}
\usepackage{amssymb}
\usepackage{amsthm}

\usepackage{graphicx}

\usepackage{ifthen}
\usepackage{color}

\newcommand{\intav}[1]{\mathchoice {\mathop{\vrule width 6pt height 3 pt depth  -2.5pt
\kern -8pt \intop}\nolimits_{\kern -6pt#1}} {\mathop{\vrule width
5pt height 3  pt depth -2.6pt \kern -6pt \intop}\nolimits_{#1}}
{\mathop{\vrule width 5pt height 3 pt depth -2.6pt \kern -6pt
\intop}\nolimits_{#1}} {\mathop{\vrule width 5pt height 3 pt depth
-2.6pt \kern -6pt \intop}\nolimits_{#1}}}

\def\polhk#1{\setbox0=\hbox{#1}{\ooalign{\hidewidth\lower1.5ex\hbox{`}\hidewidth\crcr\unhbox0}}}

\newcommand{\bd}{{\bf d}}

\usepackage{ifthen}
\usepackage{color, xcolor}
\usepackage[colorlinks=true]{hyperref}
\hypersetup{urlcolor=blue, citecolor=red}

\newtheorem{theorem}{Theorem}

\newtheorem{definition}{Definition}
\newtheorem{lemma}{Lemma}
\newtheorem{corollary}{Corollary}
\newtheorem{proposition}{Proposition}
\newtheorem{remark}{Remark}
\newtheorem{assumption}{A}

\usepackage{enumerate}

\usepackage{setspace}
\newcommand{\Tr}{\text{Tr}}

\newcommand{\N}{\mathbb{N}}

\newcommand{\rW}[2]{W^{2,#1}(B_{#2})\cap W_g^{1,#1}(B_{#2})}
\begin{document}

\title{ Two-phase free boundary problems for a class of fully nonlinear double-divergence systems}

\author{ Pêdra D. S. Andrade and Julio C. Correa}

\date{}

\maketitle

\begin{abstract}
\begin{spacing}{1.15}
In this article, we study a class of fully nonlinear double-divergence systems with free boundaries associated with a minimization problem. The variational structure of the Hessian-dependent functional plays a fundamental role in proving the existence of the minimizers and then the existence of the solutions for the system. In addition, we establish improvements in integrability for the equation in the double-divergence form. Consequently, we improve the regularity for the fully nonlinear equation in Sobolev and H\"older spaces.
\end{spacing}

\medskip


\medskip 

\end{abstract}

\vspace{.1in}
%%%%%%%%%%%%%%%%%%%%%%%%%%%%%%%%%%%%%%%%%%%%
%%%%%%%%%%%%%%%%%%%%%%%%%%%%%%%%%%%%%%%%%%%%
%%%%%%%%%%%%%%%%%%%%%%%%%%%%%%%%%%%%%%%%%%%%
\section{Introduction}

We examine the minimization problem associated with the energy 
\begin{equation}\label{functional_main}
I[u]:= \int_{B_1} [F(D^2 u)]^p + \gamma_+\max(u, 0) - \gamma_- \min(u, 0),
\end{equation}
where $B_1$ denotes the open ball in $\mathbb{R}^d$ centered at the origin with radius $1$, $u$ belongs to a suitable class of admissible functions, $p>1$, $F$ is a fully nonlinear $(\lambda, \Lambda)$-elliptic operator, and $\gamma_+, \gamma_-$ are fixed nonnegative real numbers such that $\gamma_++\gamma_->0$. We exploit the variational structure of the functional \eqref{functional_main} to prove the existence of a minimizer for \eqref{functional_main} and to derive the fully nonlinear double-divergence systems \eqref{eq_main} below. Furthermore, we establish improved integrability in Lebesgue spaces and gains of regularity in Sobolev and H\" older spaces.

Consider a class of fully nonlinear double-divergence systems with free boundaries
\begin{spacing}{1.0} 
\begin{equation}\label{eq_main}
	\begin{cases}
		F(D^2u)\,=\,m^{1/(p-1)}&\;\;\;\;\;\mbox{in}\;\;\;\;\;B_1\\
		\left(F_{ij}(D^2u)\,m\right)_{x_ix_j}\,=\,\frac{\gamma_+}{p}\chi_{\{u>0\}}- \frac{\gamma_-}{p}\chi_{\{u<0\}}&\;\;\;\;\;\mbox{in}\;\;\;\;\;B_1,
	\end{cases}
\end{equation}
\end{spacing}
\vspace{.09in}
\noindent where $p>1$,  $\gamma_+$, $\gamma_-$ are fixed nonnegative real numbers such that $\gamma_++\gamma_->0$. In \eqref{eq_main}, $F_{ij}(M)$ denotes the first derivative of $F(M)$ with respect to the entry $m_{ij}$ of the   matrix $M= (m_{ij})_{i,j = 1}^d$ and $\chi_{A}$ stands for the characteristic function of the set $A$. Observe that the system \eqref{eq_main} requires $F$ to be of class $\mathcal{C}^1$, which implies $F_{ij}(M)$ is well-defined and $(\lambda, \Lambda)$-elliptic operator. Hence $F_{ij}(M)$ is Lipschitz continuous, then Rademacher's Theorem guarantees that Definition \ref{weak_solution} below is correct.

The system \eqref{eq_main} comprises a fully nonlinear partial differential equation and an inhomogeneous elliptic double-divergence equation. These equations have been studied extensively over the past fifty years and remain a topic of significant interest. Summarizing a substantial list of references related to the above-mentioned equations is out of the scope of this paper; we refer the reader to the monographs \cite{bkrsbook} and \cite{Caffarelli-Cabre1995}. In addition, we can see the system in \eqref{eq_main} as the Euler-Lagrange equation associated with the functional in \eqref{functional_main}.

Observe that the second equation in \eqref{eq_main}  has a discontinuous right-hand side depending on the sign of $u$. That is, the equation in double-divergence form can be rewritten in the following way:
\[
\left(F_{ij}(D^2u)\,m\right)_{x_ix_j}\,=
\left\{
\begin{array}{rcl}
\frac{\gamma_+}{p} & \mbox{in} &B_1\cap \{ u>0\} \vspace{0.3cm}\\
- \frac{\gamma_-}{p} & \mbox{in} &B_1\cap \{ u<0\}.
\end{array}
\right.
\]
It is essential to highlight that the discontinuities on the right-hand side as in \eqref{eq_main} characterize a free boundary problem, {\it i.e.,} where we know the behavior of the solutions in the sets $\{ u>0\}$ and $\{ u<0\}$. However, we do not know what happens when the solutions cross the free boundary $\Gamma(u)$.
We denote the free boundary by
\[
\Gamma(u):= \partial (\{ u>0\}) \cup \partial(\{ u<0\}).
\]
For instance, the free boundary $\Gamma(u)$ can behave as illustrated in Figure 1.
\begin{figure}[!h]
\begin{center}
 \includegraphics[width= 4cm]{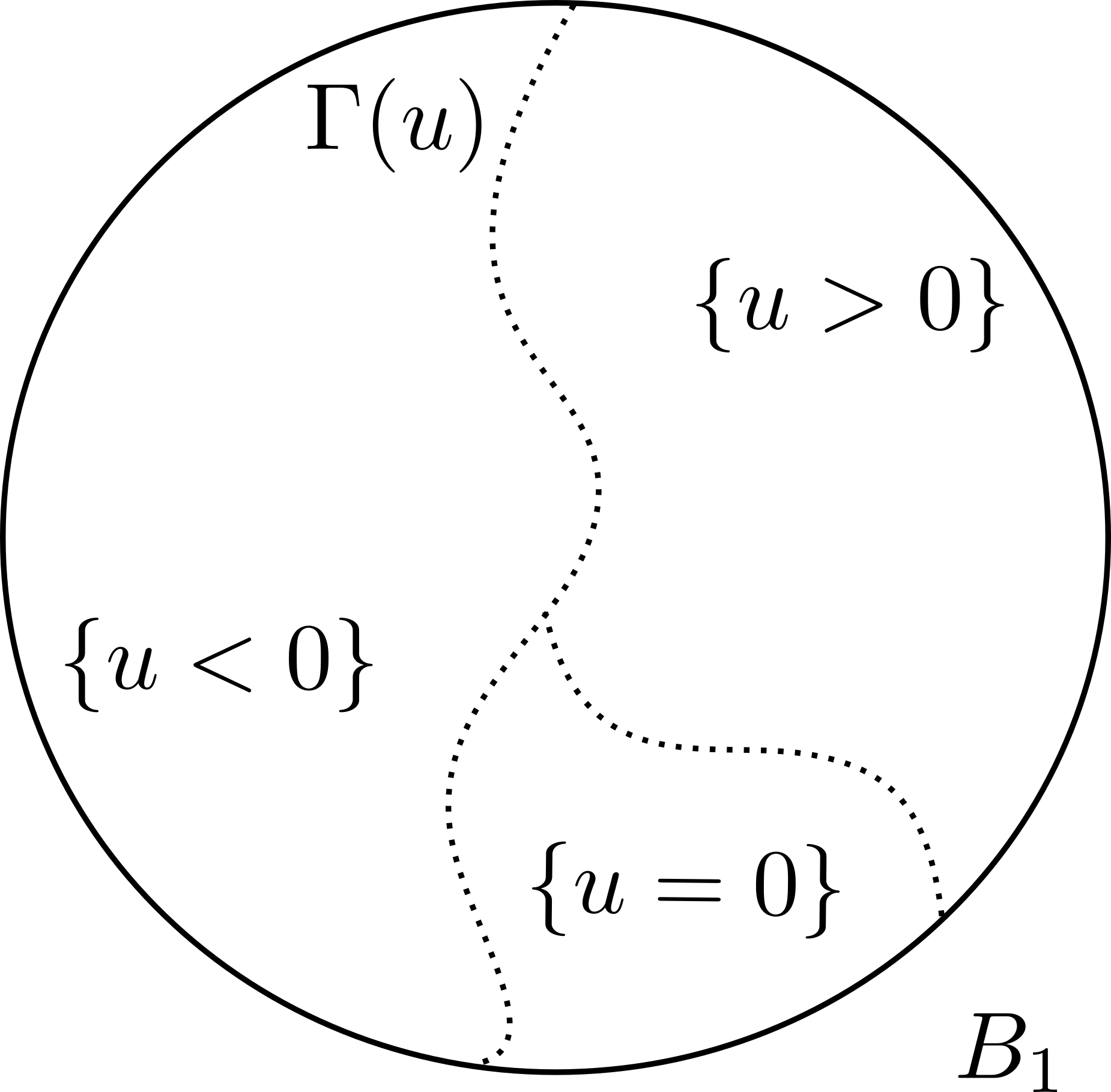}   
\end{center}
\caption{The free boundary $\Gamma(u)$ when the sets $\{ u>0\}$, $\{ u<0\}$, and $\{ u=0\}$ have positive measures.}
\end{figure}

Due to the absence of monotonicity formulas, studying the free boundary properties for this class of the energies as \eqref{functional_main} is very challenging. Another obstacle is that the techniques used to investigate the regularity of the solutions to the system \eqref{eq_main} do not allow us to establish the optimal regularity for the solutions, which is a key ingredient to studying the regularity of the free boundary. For these reasons, as far as we know, the study of the free boundary problem for \eqref{functional_main} remains an open question. 

Our analysis occurs at the intersection of several fields, such as calculus of variations, free boundary problems, and fully nonlinear double-divergence systems. The variational setting plays a crucial role since the Hessian-dependent functional combined with suitable conditions on the operator $F$ allow us to overcome the technical difficulties, such as the absence of coercivity and convexity of the functional $I$,  and the requirement that the coefficients $F_{ij}(M)$ be sufficiently regular for each $i, j \in\{1, \ldots, d \}$. Such cost functionals appear in many branches of mathematics. For instance, the Euler-Lagrange equation that is associated with the following energy
\begin{equation} \label{biharmonic_fun}
\mathcal{I}[u] := \int_{B_1} \big | \Tr(D^2 u)\big|^2 dx
\end{equation}
is the well-known biharmonic equation, where $\Tr(\cdot)$ stands for the trace of a symmetric matrix $M$. This type of energy arises from the equilibrium of thin plates. Also, in dimension $d=4$, \eqref{biharmonic_fun} is conformally invariant. We refer the reader to \cite{CGY1999} and \cite{CWY1999}. Regarding the obstacle problems for the functional \eqref{biharmonic_fun}, see \cite{CA1979} and \cite{CFT1982}.

Another important feature of the functionals depending on the Hessian is in the study of the nonconvex functionals of the form
\begin{equation} \label{noconvex_functional}
\mathcal{J}[u]:= \int_{B_1} \left(|Du|^2 - 1\right)^2 dx.
\end{equation}
Regularizing the functional in \eqref{noconvex_functional}, we obtain an approximated energy that is known in the literature as Aviles-Giga functional, which is convex with respect to higher order terms and a Hessian-dependent functional, namely
\begin{equation} \label{convex_functional}
\mathcal{J}_{\epsilon}[u]:= \int_{B_1} \left(|Du|^2 - 1\right)^2  + \epsilon^2 \|D^2 u \|^2dx.
\end{equation}
Notice that we recover the information in \eqref{noconvex_functional} as $\epsilon$ goes to zero in \eqref{convex_functional}.  Such limit leads to a phenomenon known as microstructures. For more details, see \cite{AG1987}, \cite{AG1996}; see also \cite{kohnicm}.

From the application point of view, the Hessian-dependent functionals also appear in the mechanics of solids.  For instance, the study of energy-driven pattern formation and nonlinear elasticity as in \cite{KO2018}. The ribbon’s energy is defined by 
\begin{equation*}
 E^{(h)}[u, v]:= \int_{\Omega} |{\mathcal M}(u, v)|^2 + h^2|{\mathcal B}(u, v)|^2 dx,
\end{equation*}
where $\Omega  =\left(-1/2, 1/2\right) \times(0, l)$ and   $h$ is a parameter that represents the thickness of the ribbon, and ${\mathcal M}$ (for membrane) and ${\mathcal B}$ (for bending) are symmetric tensors with physical interpretations. Here, ${\mathcal M}$ and ${\mathcal B}$ are given by
\[
\begin{aligned}
\mathcal{M} & =\mathrm{e}(u)+\frac{1}{2}\binom{\partial_1 v}{\partial_2 v+\omega x_1}^{\otimes 2}-\frac{1}{2}\left(\begin{array}{cc}
0 & \omega v \\
\omega v & \omega^2 \xi^2
\end{array}\right) \text { and }\, \, 
\mathcal{B} & = D^2 v+\left(\begin{array}{cc}
0 & \omega \\
\omega & 0
\end{array}\right),
\end{aligned}
\]
where $\mathrm{e}(u)$ is the linear strain, and the ribbon is twisted by an amount $\omega$ around the $x_2$ axis, and allowed to compress by an amount $\frac{1}{2} \omega^2 \xi^2 l$, where $0<\xi<\frac{1}{2}$.
Even though ${\mathcal M}$ depends only on lower-order terms, ${\mathcal B}$ depends on the Hessian of $v$. An interesting case is in the analysis of the behavior as $h$ approaches zero. Also, the analysis of blister patterns in thin films on compliant substrates in \cite{BK2015} is an example where higher-order functionals arise in the setting of mechanics of solids. 

In the context of the functionals depending on the Hessian with free boundaries, the authors in \cite{DH2022}, the authors showed the existence of the solutions for $\tilde{\mathcal{I}}$, which is given by
\[
 \int_{B_1^+}(\Delta u)^2 dx dy + \frac{2}{p} \int_{B_1'} \gamma_-(u^-)^p +  \gamma_+(u^+)^p dx,
\]
where ${B_1^+} = \{\, (x, y) \in \tilde{B}_1\,\, ;\, \, y>0 \}$ is the upper half-ball, $\tilde{B}_1$ is the open ball in $\mathbb{R}^{n+1}$ centered at the origin with radius $1$, $B_1' = \tilde{B}_1\cap \{ y =0 \} $, $\gamma_-, \gamma_+$ are positive constants, and $p>1$. In addition, by using the monotonicity formulas, the authors showed the optimal regularity of the solution and studied the free boundary's structure.  

More recently, the authors in \cite{Andrade-Pimentel} have studied a class of fully nonlinear mean-field games (for short, MFG) systems. They proved the existence of minimizers for the Hessian-dependent functional associated with the MFG system.  Exploiting the MFG structure, they established gains of regularity and showed that solutions to the MFG system exist. In addition, without assuming any convex property on the operator $F$, the authors obtained the existence of the minimizer for a relaxed problem. 

In \cite{CorreaPimentel2021}, the authors investigate the functional depending on the Hessian driven by a fully nonlinear operator, namely
\begin{equation} \label{funct_CP}
J[u]:= \int_{B_1} [F(D^2 u)]^p dx + \Lambda|\{u>0 \} \cap B_1|.
\end{equation}
Taking advantage of the variational setting, they established that there exists a minimizer for \eqref{funct_CP} and the existence of the solutions for the MFG systems. Also, the authors proved H\"older continuity for the solutions and improved integrability of the density. Moreover, they showed that the reduced boundary has a finite perimeter. 

 In \cite{CJK2022}, the authors have studied a parabolic version of the fully nonlinear MFG system with nonlocal and local diffusions, namely:
\begin{equation} \label{parabolicMFG}
\begin{cases}
- u_t &=F(\mathcal{L}u) + f(m) \;\;\;\mbox{on}\;\;\;\;\;(0, T)\times \mathbb{R}^d\\
m_t &=  \mathcal{L}^*(F^{\prime}(\mathcal{L}(u))m)\;\;\;\mbox{on}\;\;\;\;\;(0, T)\times \mathbb{R}^d\\
\end{cases}
\end{equation}
with initial and terminal conditions
\[
u(T)= g(m(T)) \quad \quad m(0)= m_0,
\]
where $\mathcal{L}$ stands for an infinitesimal generator of a Lévy process and $F\in \mathcal{C}^1(\mathbb{R})$ is a convex function with $F' \geq 0$ belonging to a H\"older space. The authors established the existence of the solutions to \eqref{parabolicMFG} and the uniqueness of solutions in the absence of strict monotonicity of couplings or strict convexity of Hamiltonians.
    
The paper's main contribution is to investigate the existence and regularity estimates of the solutions for a more general class of systems than the ones studied in \cite{Andrade-Pimentel} and \cite{CorreaPimentel2021}. We emphasize that, in our case, we do not have the adjoint structure as in the systems \cite{Andrade-Pimentel} and \cite{CorreaPimentel2021}. That means, we do not have the MFG structure in our framework. However, our approach is motivated by these papers. In addition,  a fundamental difference between the MFG system in \cite{Andrade-Pimentel} and the fully nonlinear double-divergence system with free boundaries in \eqref{eq_main} is the second equation for both systems.  In \cite{Andrade-Pimentel}, the second equation is homogeneous, whereas in our case, the right-hand side of the second equation is a discontinuous function that depends on the sign of $u$.

In this paper, we take advantage of the relationship between the variational formulation in \eqref{functional_main} and the fully nonlinear double-divergence system in \eqref{eq_main} to obtain improved regularity estimates and the existence of minimizers for \eqref{functional_main} in a set of the convex functions. Moreover, under an additional condition on the operator $F$, we prove the existence of minimizer in $W^{2,p}_g(B_1)\cap W^{1,p}_g(B_1)$, see Proposition \ref{Prop:Existence-of-minimizer}.
  
Furthermore, we show that solutions exist in the system \eqref{eq_main}. We report our findings in Theorem \ref{Theorem1}.  Then, as in \cite{Andrade-Pimentel}, we are able to prove gains of integrability by using the arguments in the seminal paper \cite{fabesstroock}, see Lemma \ref{lemma_int} and Lemma \ref{Gain-of-Integrability}.  Finally, we employ the Sobolev embedding Theorem combined with the gains of integrability to obtain further regularity for the solutions to the first equation of  \eqref{eq_main}.  This is the content of Theorem \ref{Gains-of-Regularity}. 

The remainder of this article is organized as follows: Section \ref{preliminaries} presents the main assumptions and gathers a few results used throughout the manuscript. In Section \ref{existence_sec}, we derive the Euler-Lagrange equation associated with the energy \eqref{functional_main}. In addition, we prove the existence of minimizers for \eqref{functional_main} and the existence of solutions to the system \eqref{eq_main}. In Section \ref{regularity_sec}, we establish gains of integrability for the solutions of the double-divergence equation. Consequently, we obtain improved regularity for the fully nonlinear equation in Sobolev
 and H\"older spaces. 
 
\section{Preliminaries} \label{preliminaries}
In this section, we detail the main assumptions and gather some auxiliary results used in the paper.

\subsection{Main assumptions}\label{subsec_assump}
Throughout this paper, $\mathcal{S}(d)$ stands for the space of real $d \times d$ symmetric matrices. In the sequel, we introduce the condition of the uniform ellipticity of the operator $F$.

\begin{assumption}[Uniform ellipticity]\label{A1}
We suppose the operator $F: \mathcal{S}(d) \rightarrow$ $\mathbb{R}$ is $(\lambda,\Lambda)$-uniformly elliptic for some $0<\lambda\leq \Lambda$. That is,
\begin{equation}\label{eq_lambdaeliptic}
	\lambda\|N\| \leq F(M+N)-F(M) \leq \Lambda\|N\|
\end{equation}
for every $M, N \in \mathcal{S}(d)$, with $N \geq 0 $. We also suppose $F(0)=0$.  
%Finally, we require $F_{ij}(M)=F_{ji}(M)$, for every $i,j=1,\ldots,d$, where $F_{ij}(M)$ stands for the derivative of $F$ with respect to the entry $(i,j)$ of $M$.
\end{assumption}

\begin{remark}\label{Remark:Assumptions}
It follows from A\ref{A1} that $F$ satisfies a coercivity condition over nonnegative symmetric matrices. Taking $M \equiv 0$ in the inequalities in \eqref{eq_lambdaeliptic}, we have
\[
	\lambda\|N\| \leq F(N) \leq \Lambda\|N\| \; \; \mbox{for every} \;  \;N \geq 0.
\] 
\end{remark}
\begin{remark} Set
	\[
	F_{ij}(M):= \dfrac{\partial F}{\partial m_{ij}}(M),
	\]
Notice that, if $F$ is an $(\lambda, \Lambda)$-uniformly elliptic operator and of class $\mathcal{C}^1$, then $F_{ij}$ satisfies 
	\begin{equation}\label{Derivative_F}
	\lambda |\xi|^2 \leq F_{ij}(M){\xi}_i{\xi}_j \leq \Lambda |\xi|^2 \quad \forall\, M\in \mathcal{S}(d)\quad \forall \,\xi \in \mathbb{R}^d.
	\end{equation}
 In particular, $F_{ij}(M)$ is Lipschitz in $M \in \mathcal{S}(d)$ for all $i, j$. 
\end{remark}

In what follows, we enforce a convexity condition on $F$.

\begin{assumption}[Convexity of the operator $F$]\label{A2}
We assume that the operator $F=F(M)$ is convex in $M \in \mathcal{S}(d)$.
\end{assumption}

It is important to highlight that our arguments to prove the existence of solutions to \eqref{eq_main} require $F$ to satisfy a coercivity condition in whole $\mathcal{S}(d)$. To that end, we extend A\ref{A1} in the following sense:
\begin{assumption}[Growth condition]\label{A3}
 Let $\lambda$ and $\Lambda$ be constants such that $0<\lambda\leq \Lambda$. We suppose that the operator $F$ satisfies
\[
	\lambda\|M\| \leq F(M) \leq \Lambda\|M\| \; \; \mbox{for every } \; \; M \in \mathcal{S}(d).
\]
\end{assumption}

The following is an example of operators that satisfy the growth condition: Let $A$ be a nonnegative function in $W_{loc}^{2,p}(B_1)$, $\varepsilon>0$ and define the operator
\[
	F(x,M):=A(x)\|M + \varepsilon I\| 
\]
where $M \in \mathcal{S}(d)$  and $I$ represents the identity matrix in $\mathcal{S}(d)$ and $x \in B_1$. In fact, it is clear that the operator $F$ satisfies growth condition in the previous assumption.

Next, we introduce an assumption on the boundary data.

\begin{assumption}[Boundary data]\label{A4}
We assume $g\in W^{2,p}(B_1)$ to be a nontrivial function.
\end{assumption}

\subsection{Auxiliary results}
We begin this subsection by presenting the notion of $L^p$-viscosity solutions.

\begin{definition}[$L^p$-viscosity solutions]\label{def_lpviscosity}
Let $F:\mathcal{S}(d)\to\mathbb{R}$ be a fully nonlinear $(\lambda, \Lambda)$-elliptic operator as defined in A\ref{A1}, $p>d/2$ and $f\in L^p_{loc}(B_1)$. A continuous function $u$ in $B_1$ is an $L^{p}$-viscosity supersolution (resp. subsolution) of 
\begin{equation}\label{eq-fullynonlinear}
	F(D^2u)=f\hspace{.2in}\mbox{in}\hspace{.2in}B_1
\end{equation}
if, for all $\varphi \in W_{\mathrm{loc}}^{2, p}(B_1)$, whenever $\varepsilon>0$, $\Omega$ is an open subset of $B_1$, and
\begin{align*}
 F\left(D^{2} \varphi(x)\right)-f(x) &\leq-\varepsilon \hspace{.2in} \mbox{a.e.}-x\in \Omega\\
(\mbox{resp. } F\left(D^{2} \varphi(x)\right)-f(x) &\geq+\varepsilon \hspace{.2in} \mbox{a.e.}-x\in \Omega),
\end{align*}
then $u-\varphi$ cannot have a local minimum (resp. maximum) in $\Omega$. Furthermore, we say that $u$ is an $L^p$-viscosity solution of \eqref{eq-fullynonlinear} when it is an $L^p$-viscosity supersolution and an $L^p$-viscosity subsolution.
\end{definition}

 The notion of $L^p$-viscosity solutions is necessary since $F$ and $f$ might not be continuous functions at the points in their respective domains. For a comprehensive account of the theory of $L^{p}$-viscosity solutions, we refer the reader to \cite{CafCraKocSwi96}.

In what follows, we introduce the definition of solution to the fully nonlinear double-divergence systems in \eqref{eq_main}.

 \begin{definition}[Weak solution]\label{weak_solution} A pair $(u,m)$ satisfying \eqref{eq_main} is called a weak solution provided
\begin{enumerate}
	\item[(i)] $u\in \mathcal{C}(B_1)$ and $m\in L^1(B_1)$ satisfies $m\geq 0$ in $B_1$;
	\item[(ii)] $u$ is an $L^p$-viscosity solution to 
	\[
		F(D^2u)\,=\,m^\frac{1}{p-1}\;\;\;\;\;\mbox{in}\;\;\;\;\;B_1;
	\]
	\item[(iii)] $m$ satisfies
	\[
		\int_{B_1}\left(F_{ij}(D^2u)\,m\right)\varphi_{x_ix_j}\,\bd x\,=\,\int_{B_1} f \varphi \,\bd x,
	\]
	for every $\varphi\in\mathcal{C}^\infty_c(B_1)$.
\end{enumerate}
\end{definition}

The energy functionals in this paper are set in appropriate Sobolev space. We proceed with a definition:
\begin{definition}[Sobolev spaces $\rW p1$]
 Let $p\in(d/2,d]$ be fixed. Given $g \in W^{2, p}\left(B_{1}\right)$, we say that $u \in \rW p1$ if $u \in W^{2, p}\left(B_{1}\right)$ and $u-g \in W_{0}^{1, p}\left(B_{1}\right)$.
\end{definition}

Our existence result relies on the weak lower semicontinuity of the functional $J$ within the class of convex functions belonging to $ W^{2,p}(B_1)\cap W^{1,p}_g(B_1)$, defined by
\begin{equation}\label{functional:fn}
J[u]:= \int_{B_1} [F(D^2 u)]^p.
\end{equation}
For clarity, we state the proposition here.
\begin{proposition}{\rm \cite[Proposition 3]{Andrade-Pimentel}} \label{prop:wlsc}
Assume that A\ref{A1} and A\ref{A2} hold true.  Let $(u_n)_{n\in\mathbb{N}}$ be a sequence of convex functions in $W^{2,p}(B_1)\cap W^{1,p}_g(B_1)$ such that 
\[
	D^2u_n\,\rightharpoonup \, D^2u_\infty\;\;\;\;\;\mbox{in}\;\;\;\;\;L^p(B_1,\mathbb{R}^{d^2}).
\]
Then, 
\[
	\liminf_{n\to\infty} J[u_n]\,\geq\,J[u_\infty].
\]
\end{proposition}

The proof of Proposition \ref{prop:wlsc} can be found in \cite[Proposition 3]{Andrade-Pimentel}. Next, we state Poincare's inequality for functions lacking compact support. This result is very useful in proving the existence of minimizers for \eqref{functional_main}. 

\begin{lemma}[Poincar\'e's inequality]\label{poincare}
Let $u \in W_{g}^{1, p}\left(B_{1}\right)$ and $C_{p}>0$ be the Poincar\'e's constant associated with $L^{p}\left(B_{1}\right)$ and the dimension $d$. Then for every $C<C_{p}$, there exists $C_{1}\left(C, C_{p}\right)>0$ and $C_{2} \geq 0$ such that
$$
\int_{B_{1}}|D u|^{p} d x-C \int_{B_{1}}|u|^{p} d x+C_{2} \geq C_{1}\left(\int_{B_{1}}|D u|^{p} d x+\int_{B_{1}}|u|^{p} d x\right).
$$
\end{lemma}
For the proof of the Lemma \ref{poincare}, we refer the reader to \cite[Lemma 2.7]{dalmaso}.

In what follows, we include Remark 3 from \cite{Andrade-Pimentel} for completeness. 
\begin{remark}\label{remark:existence-of-minimizer}
If assumption A\ref{A1} is replaced with  A\ref{A3}, the conclusion of \cite[Proposition 2]{Andrade-Pimentel} changes. In fact, under A\ref{A3}, we have
\[
F(D^2 u)\ge \lambda \|D^2 u \|\ge 0 \, \,\,  \,  \text{for every}\, \, \,\,  u \in W^{2,p}(B_1)\cap W^{1,p}_g(B_1).
\]
 As a consequence, we obtain that the functional $J$, defined as in \eqref{functional:fn}, is weakly lower-semicontinuous over $W^{2,p}(B_1)\cap W^{1,p}_g(B_1)$. Furthermore, A\ref{A3} yields 
\[
\|D^2u_n\|_{L^p(B_1)} \le \frac{1}{\lambda^p} \, \displaystyle \int_{B_1}  [F(D^2 u_n)]^p \le C
\] 
 for every minimizing sequence $\left(u_n\right)_{n \in \mathbb{N}} \subset W^{2,p}(B_1)\cap W^{1,p}_g(B_1)$ and a positive constant $C$.
 
Therefore, under the assumptions A\ref{A2} and A\ref{A3}, there exists $u^* \in W^{2,p}(B_1)\cap W^{1,p}_g(B_1)$ such that
\[
J[u^*] \le J[u] \,  \, \, \, \text{for all} \, \, \, \,  u \in W^{2,p}(B_1)\cap W^{1,p}_g(B_1).
\]
That is, the functional $J$ admits a minimizer in the space $W^{2,p}(B_1)\cap W^{1,p}_g(B_1)$.
\end{remark}

%%%%%%%%%%%%%%%%%%%%%%%%%%%%%%%%%%%%%%%%%%

\section{Existence results and the Euler-Lagrange equation} \label{existence_sec}
The main purpose of this section is to prove the existence of the solution to the system \eqref{eq_main}. Under the Assumptions A\ref{A1} and A\ref{A2}, we initiate our analysis by proving the existence of minimizers for \eqref{functional_main} within the class of convex functions. Let us consider $g \in W^{2,p}(B_1)$ be a convex function and define 
\[
 \mathcal{A}:= \{u \in W^{2,p}(B_1)\cap W^{1,p}_g(B_1) | \: u \: \mbox{is convex}\}
\]
to be the admissible set. Further, under the additional condition A\ref{A3}, we infer the existence of a minimizer in $W^{2,p}(B_1)\cap W^{1,p}_g(B_1)$. Moreover, we derive the Euler-Lagrange equation associated with \eqref{functional_main}, which leads to the fully nonlinear double-divergence systems in \eqref{eq_main}.
\subsection{Existence of Minimizers}

In what follows, we establish that \eqref{functional_main} has a minimizer among the functions of $\mathcal{A}.$ The proof of the following proposition is based on {\it direct method in the calculus of variations} that is a classical tool to prove the existence of a minimizer for a given functional.  Recall that $u^+= \max(u, 0)$, $u^-= \max(-u, 0)$ and $|u|:= \max(u, 0) + \max(-u, 0)$. Henceforth, we will use these notations.

\begin{proposition}[Existence of minimizers]\label{Prop:ExistenceMin}
    Assume A\ref{A1}, A\ref{A2} and A\ref{A4} hold true. Then there exists a minimizer $u^* \in \mathcal{A}$ for \eqref{functional_main}.
\end{proposition}

\begin{proof} 
We divide the proof into two steps. 

{\it Step 1.} First, we shall verify that the functional $I$ satisfies the coercive property. We use the fact that $u \in \mathcal{A}$ combined with the assumption A\ref{A1}, and $\gamma_+\geq 0, \gamma_-\geq 0$ such that $\gamma_+ + \gamma_->0$, we infer that
\[
\begin{array}{ccl}
I[u] &\geq &\lambda^p\|D^2u\|^p_{L^p(B_1)}+\displaystyle \int_{B_1}(\gamma_+u^+ + \gamma_-u^-)\vspace{0.25cm}\\
&\geq & \lambda^p\|D^2u\|^p_{L^p(B_1)}+ \min(\gamma_+, \gamma_-  )\displaystyle \int_{B_1}u^+ + u^-\vspace{0.25cm}\\
&\geq & \lambda^p\|D^2u\|^p_{L^p(B_1)}+ \min(\gamma_+, \gamma_-)  \displaystyle \int_{B_1}|u|\vspace{0.25cm}\\
&\geq & \lambda^p\|D^2u\|^p_{L^p(B_1)}.
\end{array}
\]
Therefore, we can conclude that $I[u]$ goes to infinity as $\|D^2u\|^p_{L^p(B_1)}$ approaches to infinity.

{\it Step 2.} In the sequel, we define
\[
\eta:=\inf_{u \in \mathcal{A}} I[u^*].
\]
Notice that $0<\eta<+\infty$ because the coercive property holds for $I[\cdot]$. Let us now consider $(u_n)_{n\in\N}\subset \mathcal{A}$ be a minimizing sequence, \emph{i.e.,} there exists $N\in\N$ such that for every $n> N$
\[
I[u_n]\le\eta+1.
\]
Hence, 
\[
\|D^2u_n\|_{L^p(B_1) }^p \leq \frac1{\lambda^p}\int_{B_1}[F(D^2u_n)]^p+\frac1{\lambda^p}\int_{B_1} \gamma_+u_n^++\gamma_-u_n^- \leq C, 
\]
where $C>0$ is a constant depends on $\lambda, p$ and $\eta$. Therefore, we can conclude that $(D^2u_n)_{n\in\N}\subset L^p(B_1)$ is uniformly bounded.

 Furthermore, invoking the Lemma \ref{poincare}, we obtain $(u_n)_{n\in\N}$ is uniformly bounded in $W^{2,p}(B_1)\cap W^{1,p}_g(B_1)$. Upon passing to a subsequence, there exists $u_{\infty}\in \mathcal{A}$ such that
\begin{equation}\label{convergencemode1}
u_n\rightharpoonup u_\infty\hspace{.2in}\mbox{in}\hspace{.1in} W^{2,p}(B_1)\cap W^{1,p}_g(B_1)
\end{equation}
and
\begin{equation}\label{convergencemode2}
u_n\to u_{\infty}\hspace{.2in}\mbox{strongly in}\hspace{.1in} L^p(B_1).
\end{equation}
Now, we can rewrite $I$ in the following way:
\begin{equation}\label{sum-of-functionals}
I[u]=J[u]+\int_{B_1}\gamma_+u^++\gamma_-u^-,
\end{equation}
where 
\begin{equation}\label{Andrade-Pimentel_functional}
J[u]: = \int_{B_1}[F(D^2u)]^p.
\end{equation}
Observe that the fact that $J$ is weakly lower semicontinuous follows from the Proposition \ref{prop:wlsc}. Hence,  we only need to verify that the inequality below holds
\begin{equation}\label{existence2}
\int_{B_1}\gamma_+u_\infty^++\gamma_-u_\infty^- \le\liminf_{n\to\infty}\int_{B_1}\gamma_+u_n^++\gamma_-u_n^-.
\end{equation}
We deduce from \eqref{convergencemode2} that $u_n$ converges to $u_{\infty}$  almost everywhere in $B_1$. Hence, we can employ Fatou's Lemma to get the desired inequality in \eqref{existence2}. Finally, we conclude that the functional $I$ is weakly lower semicontinuous and therefore
\[
I[u_{\infty}] \leq \liminf_{n\to\infty} I[u_n] = \eta
\]
Therefore, by taking $u^*:=u_{\infty}$, we complete the proof.
\end{proof}

\begin{remark}
Observe that the Proposition \ref{Prop:ExistenceMin} yields a minimizer among convex functions; there is no reason for this minimizer to be a critical point in $W^{2,p}(B_1)\cap W^{1,p}_g(B_1)$. Replacing A\ref{A1} with A\ref{A3}, we obtain the existence of minimizers in the entire space $W^{2,p}(B_1)\cap W^{1,p}_g(B_1)$. This is the content of the following proposition.
\end{remark} 

\begin{proposition}\label{Prop:Existence-of-minimizer}
Assume A\ref{A2} - A\ref{A4} hold true. Then, there exists $u^* \in W^{2,p}(B_1)\cap W^{1,p}_g(B_1)$ so that
\[
I[u^*]\leq I[u] \quad \text{for\, all} \;\;  u \in W^{2,p}(B_1)\cap W^{1,p}_g(B_1).
\]
\end{proposition}
\begin{proof}
Under A\ref{A3}, the proof of the coercivity follows the same lines as in Step $1$ in the proof of Proposition \ref{Prop:ExistenceMin}. It remains to check that $I$ is weakly lower semicontinuous in $W^{2,p}(B_1)\cap W^{1,p}_g(B_1)$. Recall that we can rewrite $I$ as in the equation \eqref{sum-of-functionals}. Hence, by using Proposition \ref{prop:wlsc} combined with Remark \ref{remark:existence-of-minimizer} and arguing as in Step $2$ in  the proof of Proposition \ref{Prop:ExistenceMin}, we obtain the desired property in $W^{2,p}(B_1)\cap W^{1,p}_g(B_1)$, which finishes the proof.
\end{proof}

%%%%%%%%%%%%%%%%%%%%%%%%%%%%%%%%%%%%%%%%%%%%
 %%%%%%%%%%%%%%%%%%%%%%%%%%%%%%%%%%%%%%%%%%%%
\subsection{Derivation of the fully nonlinear double-divergence system}
This section is devoted to deducing the {\it Euler-Lagrange equation} associated with the functional \eqref{functional_main}.
\begin{proposition}\label{Proposition:EulerLagrange}
We assume A\ref{A2} - A\ref{A4} hold. Further, suppose $u\in W^{2,p}(B_1)\cap W^{1,p}_g(B_1)$ be a minimizer of \eqref{functional_main} and $\gamma_+$ and $\gamma_-$ be fixed nonnegative real constants such that $\gamma_++\gamma_->0$. Then, the Euler-Lagrange equation associated to the minimization problem \eqref{functional_main} is given by 
\begin{equation}\label{Euler-Lagrange}
\left([F(D^2u)]^{p-1}F_{ij}(D^2u)\right)_{x_ix_j}=\frac{\gamma_+}p\chi_{\{u>0\}}-\frac{\gamma_-}p\chi_{\{u<0\}}\hspace{.2in}\mbox{a.e. in}\hspace{.1in} B_1,
\end{equation}
in the distributional sense.
\end{proposition}

\begin{proof}
First, we write
\[
u^+:=u\chi_{\{u>0\}}\hspace{0.2in}\mbox{and}\hspace{0.2in}u^-:=-u\chi_{\{u<0\}}.
\]
With the use of A\ref{A3} and $0<\beta:= \gamma_+ +\gamma_-$, we have
\[
|[F(D^2u)]^{p} + \gamma_+ u^+ + \gamma_- u^- | \leq \Lambda^p \| D^2u\|^p + \beta|u|,
\]
It follows from the inequality above that $I[u^*+ t \varphi]$ is well-defined, for every $\varphi\in \mathcal{C}_c^{\infty}(B_1)$ and for every $t \in \mathbb{R}$, and by the minimality of $u^*$, we have
\[
 I[u^*] \leq I[u^*+ t \varphi] \: \: \text{for every}\: \: \varphi\in \mathcal{C}_c^{\infty}(B_1).
\]
Hence, we use the definition of $J$ in \eqref{Andrade-Pimentel_functional} to rewrite the inequality above as follows:
\begin{align*}
0 &\leq \frac{J[u^*+ t \varphi]-J[u^*]}{t}+ \gamma_+\int_{B_1} \frac{(u^*+t\varphi)^{+}-({u^*})^{+}}{t}  \: dx +\gamma_- \int_{B_1} \frac{(u^*+t\varphi)^{-}-({u^*})^{-}}{t}  \: dx\\
& = \mathtt{A} + \gamma_+ \mathtt{B} + \gamma_- \mathtt{C}.
\end{align*}
Next, we estimate the value of $\mathtt{C}$.
\[
{\mathtt C} = \frac{1}{t} \int_{B_1} (u^*+t\varphi)^{-}- ({u^*})^{-} \: dx = \frac{1}{t} \int_{B_1 \cap \{u^*+t\varphi<0\}} -(u^*+t\varphi)  \: dx - \frac{1}{t} \int_{B_1 \cap \{u^*<0\}} - u^*  \: dx.
\]
Note that
\[
\{u^*<0\} \cap B_1 = \left(\{ u^* + t \varphi<0 \} \cup \{ -t\varphi\le u^* < 0\} \backslash \{0 \le u^* <-t \varphi \}\right) \cap B_1,
\]
then, we obtain the estimate of $\mathtt{C}$ as follows:
\begin{align*}
{\mathtt C} &= \frac{1}{t} \int_{B_1 \cap \{u^*+t\varphi<0\}} -(u^*
+t\varphi)  \: dx + \frac{1}{t} \int_{B_1 \cap \{ u^* + t \varphi<0 \}}u^*  \: dx + \frac{1}{t} \int_{B_1 \cap \{ -t\varphi\le u^* < 0\}}u^*  \: dx - \frac{1}{t} \int_{B_1 \cap \{0 \le u^* <-t \varphi \}}u^*  \: dx \\
& =  \int_{B_1 \cap \{u^*+t\varphi<0\}} - \varphi  \: dx + \frac{1}{t} \int_{B_1 \cap \{ -t\varphi\le u^* < 0\}}u^*  \: dx - \frac{1}{t} \int_{B_1 \cap \{0 \le u^* <-t \varphi \}}u^*  \: dx \\
& \le \int_{B_1 \cap \{u^*+t\varphi<0\}} - \varphi \: \end{align*}
 A similar reasoning yields an estimate of the value of $\mathtt{B}$ 
\[
\mathtt{B} \le \int_{B_1 \cap \{u^*+t\varphi>0\}}  \varphi \: dx.
\]
Finally, we can conclude that
\[
 \frac{J[u^*+ t \varphi]-J[u^*]}{t}+ \gamma_+ \int_{B_1 \cap \{u^*+t\varphi>0\}}  \varphi \: dx - \gamma_- \int_{B_1 \cap \{u^*+t\varphi<0\}} \varphi \: dx \ge 0.
\]

By employing the limit as $t$ goes $0$, we have that the first term on the left-hand side is given by
\[
\lim_{t\to 0}\frac{J[u^*+t\varphi]-J[u^*]}{t}=\int_{B_1}\left(p[F(D^2u^*)]^{p-1}F_{ij}(D^2u^*)\right)_{x_ix_j}\varphi.
\]
Observe that the integral above is well-defined: indeed, by integrating by parts twice and using the fact that $F$ and $F_{ij}$ are $(\lambda, \Lambda)$-uniformly elliptic operators together with  $u \in W^{2, p}(B_1)$, we obtain 
\begin{equation} \label{L1-function}
  \left([F(D^2u^*)]^{p-1}F_{ij}(D^2u^*)\right)_{x_ix_j}\in L^1_{\mathrm{loc}}(B_1).  
\end{equation}
Note that the equality above follows from \cite[Theorem 1]{Andrade-Pimentel}. Therefore, for all $\varphi \in \mathcal{C}^{\infty}_c$, we have
\[
\int_{B_1}\left(p[F(D^2u^*)]^{p-1}F_{ij}(D^2u^*)\right)_{x_ix_j}\varphi + \gamma_+ \int_{B_1 \cap \{u^*>0\}}  \varphi \: dx - \gamma_- \int_{B_1 \cap \{u^*<0\}} \varphi \: dx \ge 0.
\]
Combining \eqref{L1-function} with a common density argument, we can show that 
\[
\varphi\mapsto\int_{B_1}\varphi\left(\left( p[F(D^2u^*)]^{p-1}F_{ij}(D^2u^*)\right)_{x_ix_j}+\gamma_+\chi_{\{u^*>0\}}-\gamma_-\chi_{\{u^*<0\}}\right)
\] 
is a well-defined and non-negative distribution. Since 
\[
\int_{B_1}\varphi\left(\left(p[F(D^2u^*)]^{p-1}F_{ij}(D^2u^*)\right)_{x_ix_j}+\gamma_+\chi_{\{u^*>0\}}-\gamma_-\chi_{\{u^*<0\}}\right) \geq 0,
\]
for all $\varphi \in \mathcal{C}_c^{\infty}(B_1)$. In particular, it also holds for $\varphi= -\phi$, with $\phi \in \mathcal{C}_c^{\infty}(B_1)$. Hence,
\[
\int_{B_1}\phi\left(\left(p[F(D^2u^*)]^{p-1}F_{ij}(D^2u^*)\right)_{x_ix_j}+\gamma_+\chi_{\{u^*>0\}}-\gamma_-\chi_{\{u^*<0\}}\right) \leq 0.
\]
Therefore
\begin{equation*}\label{p-fullynonlinear}
  \left([F(D^2u^*)]^{p-1}F_{ij}(D^2u^*)\right)_{x_ix_j}=\frac{\gamma_+}p\chi_{\{u^*>0\}}-\frac{\gamma_-}p\chi_{\{u^*<0\}}\hspace{.2in}\hspace{.1in}\mbox{a.e. in}\hspace{.1in} B_1,  
\end{equation*}
in the distributional sense, which concludes the proof.
\end{proof}
At this point, we choose $m:= [F(D^2u)]^{p-1}$ and obtain the fully nonlinear double-divergence system as in \eqref{eq_main}.

%%%%%%%%%%%%%%%%%%%%%%%%%%%%%%%%%%%%%%%%%%%%
%%%%%%%%%%%%%%%%%%%%%%%%%%%%%%%%%%%%%%%%%%%%

\subsection{Existence of solutions}
In this section, we show that the existence of minimizers for \eqref{functional_main} guarantees that there exists a pair $(u^*, m^*)$ solving the fully nonlinear double-divergence system in \eqref{eq_main}. In order to prove the existence of solutions to \eqref{eq_main}, we have to show that there exists a $L^p$-viscosity solution  $u \in \mathcal{C}(B_1)$ satisfying 
\begin{equation*}\label{eq.1}
F(D^2u)= m^{\frac{1}{p-1}} \quad \text{in} \quad B_1,
\end{equation*}
and there exists a weak solution $0\leq m \in L^1(B_1)$  satisfying 
\begin{equation*}\label{eq.2}
		\int_{B_1}\left(F_{ij}(D^2u)\,m\right)\varphi_{x_ix_j}\,\,=\,\int_{B_1} \varphi \left(\gamma_+\chi_{\{u>0\}}- \gamma_-\chi_{\{u<0\}}\right),
\end{equation*}
	for every $\varphi\in\mathcal{C}^\infty_c(B_1)$.

\begin{theorem}\label{Theorem1}
    Suppose A\ref{A2} - A\ref{A4} hold. Then, there exists a solution $(u^*, m^*)$ to the fully nonlinear double-divergence system \eqref{eq_main}, where $u^* \in W^{2,p}(B_1)\cap W^{1,p}_g(B_1)$  is a minimizer to the functional $I$ in \eqref{functional_main}.
\end{theorem}
\begin{proof}

Applying Proposition \ref{Prop:Existence-of-minimizer}, we have  $u^*\in W^{2,p}(B_1)\cap W^{1,p}_g(B_1)$ is a minimizer for \eqref{functional_main}, which implies that there exists a negligible set $\mathcal{N}\subset B_1$ such that $D^2u^*(x)$ is well-defined for every $x\in B_1\setminus\mathcal{N}$. Hence, using A\ref{A3} again, we get that
\begin{equation}\label{almosteverywhere}
F(D^2u^*(x))\ge 0 \; \; \mbox{for a.e.}\; \;x\in B_1,
\end{equation}
that is, $u^*$ is a strong supersolution to $F(D^2 u^*) = 0$. Thus, in virtue of \cite[Lemma 2.6 ]{CafCraKocSwi96}, $u^*$ satisfies
\begin{equation} \label{Lpviscosity}
F(D^2u^*)\ge 0 \: \; \text{in the $L^p$-viscosity sense}.
\end{equation}

On the other hand, it follows from Proposition \ref{Proposition:EulerLagrange} that
\begin{equation}\label{Eq1:existencesols}
    \int_{B_1}\left([F(D^2u^*)]^{p-1}F_{ij}(D^2u^*)\right)\varphi_{x_ix_j}=\int_{B_1}\varphi\left(\frac{\gamma_+}{p}\chi_{\{u^*>0\}}-\frac{\gamma_-}{p}\chi_{\{u^*<0\}}\right)\: \: \text{ for all} \: \: \varphi\in \mathcal{C}^\infty_c(B_1).
\end{equation}

Set $m^* := [F(D^2u^*)]^{p-1}$. In view of \eqref{almosteverywhere} and \eqref{Lpviscosity}, we have that $m^*(x)$ is well-defined and satisfies $m^*(x)\ge0$ for almost everywhere $x\in B_1$. In addition, from A\ref{A3}, we have $\int_{B_1}m ^*\le C$, for some constant $C>0$, that is, $m^*\in L^1_{\mathrm{loc}}(B_1)$. Therefore, 
\[
\int_{B_1}\left(F_{ij}(D^2u^*)m^*\right)\varphi_{x_ix_j} =\int_{B_1}\varphi\left(\frac{\gamma_+}{p}\chi_{\{u^*>0\}}-\frac{\gamma_-}{p}\chi_{\{u^*<0\}}\right) \: \: \text{for every} \: \: \varphi\in\mathcal{C}^\infty_c(B_1).
\]
This means that $m^*$ is a weak solution to the system's second equation of the system \eqref{eq_main}. Finally, we claim that $u^*\in W^{2,p}(B_1)\cap W^{1,p}_g(B_1)$ is an $L^p$-viscosity solution to the first equation of \eqref{eq_main}. More precisely, $u^*$ solves 
\[
F(D^2u^*)={(m^*)}^{1/p-1} \: \: \text{in} \: \: B_1,
\]
in the $L^p$-viscosity sense. Since $F(D^2u^*(x))={(m^*)}^{1/p-1} (x)$ almost every $x\in B_1$, we can conclude that $u^*$ is a strong solution. As before, an application of \cite[Lemma 2.6 ]{CafCraKocSwi96} yields the claim.
\end{proof}
\begin{remark}
 Our results can be generalize by replacing $V_0(u)= \gamma_+ \max(u, 0) - \gamma_- \min(u, 0)$ by a Lipschitz potential $V (u)$ bounded from below. 
\end{remark}
%%%%%%%%%%%%%%%%%%%%%%%%%%%%%%%%%%%%%%%%%%%%

%%%%%%%%%%%%%%%%%%%%%%%%%%%%%%%%%%%%%%%%%%%%

\section{Improved regularity for the solutions of the system}\label{regularity_sec}
This section is devoted to the results related to gains of integrability and improved regularity for the system \eqref{eq_main}. As usual, constants standing for $C$ may change from line to line, and depend only on the suitable quantities. In addition, we will call a universal constant when a constant depends only on the ellipticity constants $\lambda, \Lambda$, and the dimension $d$.  First, we establish an interior estimate for the nonnegative supersolutions of the second equation in \eqref{eq_main}.

\begin{lemma}\label{lemma_int}
Let $m \in L^1(B_1)$ be a nonnegative weak solution to
	\begin{equation}\label{second_ineq}
	\left(F_{ij}(D^2u)\,m \right)_{x_ix_j} \leq \tilde{f } \hspace{.2in}\mbox{in}\hspace{.1in} B_2.
	\end{equation}
Assume A\ref{A1} holds and $\tilde{f}$ is a bounded function in $B_2$, we have the estimate
	\[
	\|m\|_{L^1(B_1)}\leq C\| m\|_{L^1(B_{1/2})} + C\|\tilde{f} \|_{L^{\infty}(B_2)},
	\]
the constant $C$ depending only on the ellipticity constants and $d$.
\end{lemma}
%%%%%%%%%%%%%%%%%%%%%%%%%%%%%%%%%%%%%%%%%%%%
%%%%%%%%%%%%%%%%%%%%%%%%%%%%%%%%%%%%%%%%%%%%
%%%%%%%%%%%%%%%%%%%%%%%%%%%%%%%%%%%%%%%%%%%%

\begin{proof}
Given $\delta>0$, we define the function $h(x):= [(1+ \delta)^2 - |x|^2]^2.$ Consider $B_{(1-\delta)} \subset B_{1} \subset B_{(1+\delta)}$ concentric at the origin. Computing the second derivatives of $h$ combined with \eqref{Derivative_F}, we have that 
\[
 F_{ij}(D^2u)h_{x_ix_j} \ge 8\lambda|x|^2-4d\Lambda[(1+\delta)^2-|x|^2].
\]
 Hence, for $\delta$ small enough, we can conclude that
\begin{align*}
F_{ij}(D^2u)h_{x_ix_j}\ge0\hspace{.2in}&\mbox{for}\hspace{.1in}1-\delta\le|x|<1+\delta\\
F_{ij}(D^2u)h_{x_ix_j}\ge C\hspace{.2in}&\mbox{for}\hspace{.1in}1-\delta\le|x|\le1
\end{align*}
and
\[
\left|F_{ij}(D^2u)h_{x_ix_j}\right|\le C,
\]
where $C$ is a universal constant. Hence
	\begin{equation*}
	 \int_{B_1\setminus B_{(1-\delta)}}m \leq\, C\int_{B_1\setminus B_{(1-\delta)}} m\, F_{ij}(D^2 u)h_{x_i x_j}\hspace{0.2in}
	\end{equation*}	
Since $B_1\backslash B_{(1-\delta)} \subset B_{(1 + \delta)}\backslash B_{(1-\delta)}$, we have $|B_1\backslash B_{(1-\delta)}| \leq |B_{(1 + \delta)}\backslash B_{(1-\delta)}|$. Hence 
	\[
	 \int_{B_1\backslash B_{(1-\delta)}}m \leq\, C\int_{B_{(1+ \delta)} \backslash  B_{(1-\delta)}} m\,F_{ij}(D^2 u)h_{x_i x_j}.
	\]

On the other hand, $m$ solves
\begin{equation*}
	\left(F_{ij}(D^2u)\,m \right)_{x_ix_j} \leq \tilde{f } \hspace{.2in}\mbox{in}\hspace{.1in} B_2.
\end{equation*}
 Since $h $ is regular enough and nonnegative, we obtain from the previous inequality
\begin{equation}\label{ineq:weak-solution}
	\left(F_{ij}(D^2u)\,m \right)_{x_ix_j}\, h \leq \tilde{f }\, h \hspace{.2in}\mbox{in}\hspace{.1in} B_2.
\end{equation}
By integrating inequality \eqref{ineq:weak-solution} over $B_{(1+ \delta)}$ and applying integration by parts, we obtain
\[
\int_{B_{(1+ \delta)}} m\, F_{ij}(D^2 u)h_{x_i x_j} \,\leq\, \int_{B_{(1+ \delta)} } \tilde{f} h,
\]
hence, 
\[
\int_{B_{(1+ \delta)} \backslash  B_{(1-\delta)}} m\, F_{ij}(D^2 u)h_{x_i x_j} \,\leq\, {C}\| \tilde{f}\|_{L^{\infty}(B_2)} - \int_{B_{(1- \delta)}} m\, F_{ij}(D^2 u)h_{x_i x_j},
\]
which implies that
	\[
\int_{B_{(1+ \delta)} \backslash  B_{(1-\delta)}} m\, F_{ij}(D^2 u)h_{x_i x_j} \,\leq\, \int_{B_{(1- \delta)}} m\, |F_{ij}(D^2 u)h_{x_i x_j}| + {C}\| \tilde{f}\|_{L^{\infty}(B_2)},
\]
\emph{i. e.},
\[
\int_{B_{(1+ \delta)} \backslash  B_{(1-\delta)}} m\, F_{ij}(D^2 u)h_{x_i x_j} \,\leq\, C\int_{B_{(1- \delta)}} m + {C}\| \tilde{f}\|_{L^{\infty}(B_2)},
\]
thus
	\[
 \int_{B_1}m\, \leq \, C\int_{B_{(1-\delta)}} m +  C\| \tilde{f}\|_{L^{\infty}(B_2)}.
	\]
Using an iteration argument, we obtain
	\[
	\int_{B_1}m\, \leq \, C\int_{B_{1/2}} m + C \| \tilde{f}\|_{L^{\infty}(B_2)}.
	\]
 This finishes the proof.
\end{proof}
%%%%%%%%%%%%%%%%%%%%%%%%%%%%%%%%%%%%%%%%%%%%
%%%%%%%%%%%%%%%%%%%%%%%%%%%%%%%%%%%%%%%%%%%%
%%%%%%%%%%%%%%%%%%%%%%%%%%%%%%%%%%%%%%%%%%%%
Before we continue to simplify the calculations, we introduce an operator $\mathcal{L}$ that is given by
	\begin{equation}\label{op:L}
	\mathcal{L}v:= F_{ij}(D^2 u){v}_{x_i x_j},
	\end{equation}
 where $F_{ij}(M)$ denotes the first derivative of $F(M)$ with respect to the entry $m_{ij}$ of the matrix $M=(m_{ij})_{i,j=1}^d$ and satisfies \eqref{Derivative_F}.
 
\begin{lemma}\label{Gain-of-Integrability}
Let $m \in L^1(B_{1/2})$ be a non-negative weak solution to 
	\begin{equation} \label{eq:weak_solution_m}
	\left(F_{ij}(D^2u)\,m \right)_{x_ix_j} = \tilde{f } \;\;\;\;\mbox{in}\;\;\;\; B_1.
	\end{equation}
Suppose A\ref{A1} holds true and $\tilde{f}$ is a bounded function in $B_1$. Then in fact $m \in L^{d/{d-1}}(B_1)$ and we obtain the estimate
 	\begin{equation}\label{est:gain-of-integrability}
 	\|m\|_{L^{d/{d-1}}(B_{1/2})} \leq C_1\|m\|_{L^1(B_{1/2})} + C_2\|\tilde{f}\|_{L^{\infty}(B_1)},
 	\end{equation}
where $C_1>0$ and $C_2>0$ are universal constants.
\end{lemma}
%%%%%%%%%%%%%%%%%%%%%%%%%%%%%%%%%%%%%%%%%%%%
%%%%%%%%%%%%%%%%%%%%%%%%%%%%%%%%%%%%%%%%%%%%
%%%%%%%%%%%%%%%%%%%%%%%%%%%%%%%%%%%%%%%%%%%%

\begin{proof}
To estimate the $L^{\frac{d}{d-1}}$-norm of $m$ over $B_{1/2}$, we compute 
	\[
	\sup \left\{ \int_{B_{1/2}} m \, f \right\}, 
	\]
where the supremum is taken over the set of functions $f \in \mathcal{C}^{\infty}_c(B_{1/2})$, where $f$ is a non-negative function in $B_{1/2}$ and $\|f \|_{L^d(\mathbb{R}^d)} \leq 1$.

Since the coefficients $F_{ij}$ are locally Lipschitz continuous with respect to the space $\mathcal{S}(d)$, we may apply the Theorem 9.15 and the Theorem 9.19 in \cite{GilTru01} to obtain the existence of a solution $z$ to
\begin{equation}\label{eq_z}
	\begin{cases}
		\mathcal{L} z\,=\,f&\;\;\;\;\;\mbox{in}\;\;\;\;\;B_1\\
		z =\,0&\;\;\;\;\;\mbox{on}\;\;\;\;\;\partial B_1,
	\end{cases}
\end{equation}
where the operator $\mathcal{L}$ is defined as in \eqref{op:L} and $z$ belongs to $\mathcal{C}^{2, \alpha}$.

We take a function $\phi \in \mathcal{C}_c^{\infty}(B_{3/4})$ such that $\phi \equiv 1$ in $ B_{1/2}$ and  $\big| {\partial}^{k}\phi/\partial x^{k} \big| \leq C$. Since $z$ and $\phi$ are smooth functions, we have
	\begin{equation}\label{eq_operatorL}
	\mathcal{L}(z\phi) = \phi(\mathcal{L} z) + z(\mathcal{L}\phi) + 2
  F_{ij}(D^2 u)\dfrac{\partial z}{\partial x_i}\cdot	\dfrac{\partial \phi}{\partial x_j}.
  \end{equation}
  
  It follows from the fact $f \in \mathcal{C}^{\infty}_c(B_{1/2})$ and $\phi \equiv 1$ in $ B_{1/2}$,
	\[	
	\int_{B_{1/2}} m f = \int_{B_{1/2}} m f \phi = \int_{B_1} m \phi(\mathcal{L}z).
	\]

On the other hand, multiplying \eqref{eq_operatorL} by $m$, and integrating with respect to $B_1$, we obtain
	\[
	\int_{B_1} m \phi(\mathcal{L}z) = \int_{B_1} m \mathcal{L}(z \phi) - \int_{B_1} m z (\mathcal{L}\phi) - 2 \int_{B_1}m \left[ F_{ij}(D^2 u)\dfrac{\partial z}{\partial x_i}\cdot	\dfrac{\partial \phi}{\partial x_j}\right].
	\]
 and so
	\begin{equation}\label{inequality_main}
	\int_{B_{1/2}} m f = \int_{B_1} m \mathcal{L}(z \phi) - \int_{B_1} 		m z (\mathcal{L}\phi) - 2 \int_{B_1}m \left[ F_{ij}(D^2 	u)\dfrac{\partial z}{\partial x_i}\cdot	\dfrac{\partial \phi}		{\partial x_j}\right].
	\end{equation}
	
	We estimate the expression on the right hand side of \eqref{inequality_main}. We start with the first term. From the definition of $\mathcal{L}$ and weak solution, we have that
	\[
	\int_{B_1} m \mathcal{L}(z \phi) = \int_{B_1} m F_{ij}(D^2 u)(z\phi)_{x_ix_j} =\int_{B_1} \tilde{f} z \phi,
	\]
and therefore,
\begin{equation}\label{eq_I}
\int_{B_1} m \,\mathcal{L}(z \phi) \leq C \|\tilde{f}\|_{L^{\infty}(B_1)}.
\end{equation}	

Using the definition of $\mathcal{L}$, the fact that $\phi \in \mathcal{C}_c^{\infty}(B_{3/4})$ and the second derivatives of $\phi$ are bounded by $C$ and $F_{ij}$ satifies \eqref{Derivative_F} for all $i, j$, we obtain
\begin{equation}\label{eq_II}
	\int_{B_1} m z (\mathcal{L}\phi) \leq C\int_{B_{3/4}} m.
\end{equation}

Note that $|\partial z/\partial x_i |\leq\| \nabla z\|$ for every $i$. By applying the triangle's and Cauchy-Schwarz inequalities, we get that
    \[
	\int_{B_1}m \left[ F_{ij}(D^2 u)		\dfrac{\partial z}{\partial x_i}\cdot	\dfrac{\partial \phi}{\partial x_j}\right] \leq C\int_{B_{3/4}} m \| \nabla z\|.
	\]
and consequently, using H\"older's inequality on the right-hand side, we compute
\begin{equation}\label{eq_III}
	\int_{B_1}m \left[ F_{ij}(D^2 u)		\dfrac{\partial z}{\partial x_i}\cdot	\dfrac{\partial \phi}{\partial x_j}\right] \leq C\left( \int_{B_{3/4}} m \right)^{1/2}\cdot \left(\int_{B_1} m \|\nabla z\|^2\right)^{1/2}
\end{equation}

According to the fact first noticed in the proof of Theorem 2 in \cite{Evans1985} to estimate the gradient  of $z.$
\begin{equation}\label{est:gradient-z}
\int_{B_1} m \|\nabla z\|^2 \leq C \int_{B_1} m \left(F_{ij}(D^2u) z_{x_i}z_{x_j} \right) = C  \int_{B_1} m\left(F_{ij}(D^2 u)(z^2)_{x_ix_j} - 2 fz\right).
\end{equation}
By using the chain rule, we compute $(z^2)_{x_ix_j}$, and hence 
	\begin{equation}\label{chainrule:application}
C \int_{B_1} m \left(F_{ij}(D^2u) z_{x_i}z_{x_j} \right) = C  \int_{B_1} m\left(F_{ij}(D^2 u)(z^2)_{x_ix_j} - 2 fz\right).
	\end{equation}
In view of \eqref{est:gradient-z} and \eqref{chainrule:application}, we get that
\[
\int_{B_1} m \|\nabla z\|^2 \leq C  \int_{B_1} m\left(F_{ij}(D^2 u)(z^2)_{x_ix_j} - 2 fz\right).
\]
Applying integration by parts combined with the fact that $m$ is a weak solution of \eqref{eq:weak_solution_m} and  $z^2|_{\partial B_1} = 0$, $\nabla(z^2)|_{\partial B_1} = 0$, we obtain
 	\[
 	\int_{B_1} (m F_{ij}(D^2 u))(z^2)_{x_ix_j} \, = \int_{B_1}\tilde{f} z^2\,
 	\]
and so 
\begin{equation}\label{eq_IV}
	\int_{B_1} m \|\nabla z\|^2 \leq C \|\tilde{f} \|_{L^{\infty}(B_1)} +  C\int_{B_1} m f.
\end{equation}
Combining now \eqref{inequality_main}-\eqref{eq_IV} yield
	\[
	\int_{B_{1/2}}m f\leq C\int_{B_{3/4}} m + C \|\tilde{f} \|_{L^{\infty}(B_1)}.
	\]
 
 Therefore
	\[
	\left[\int_{B_{1/2}}m^{d/{d-1}}\right]^{{d-1}/d}\leq C\int_{B_{3/4}} m + C \|\tilde{f} \|_{L^{\infty}(B_1)}.
	\]
Since $m \geq 0$, we have $\displaystyle\int_{B_{3/4}}  m \leq \displaystyle\int_{B_1}  m .$ Therefore 
	\[
	\left[\int_{B_{1/2}}m^{d/{d-1}}\right]^{{d-1}/d}\leq C\int_{B_1} m + C \|\tilde{f} \|_{L^{\infty}(B_1)}.
	\]
Applying Lemma \ref{lemma_int} on the left hand-side of the previous inequality, we get 
\[
\left[\int_{B_{1/2}}m^{d/{d-1}}\right]^{{d-1}/d}\leq C_1\| m \|_{B_{1/2}}+ C_2 \|\tilde{f} \|_{L^{\infty}(B_1)}.
\]
This completes the proof.
\end{proof}

\begin{remark}
In the proofs of the Lemmas \ref{lemma_int} and  \ref{Gain-of-Integrability}, we adapt the arguments used by  E. Fabes and D. Stroock in  \cite{fabesstroock} for the case of the right-hand side $\tilde{f}$ be discontinuous and bounded function.
\end{remark}

It is worth noticing that $L^p$-viscosity solutions to
\[
F(D^2 u) = \mu \: \: \text{in} \: \: B_1,
\]
with $\mu \in L^p(B_1)$ and $d<p$ belong to $W^{2, p}_{loc}(B_1)$, see \cite[Lemma 3.1 ]{CafCraKocSwi96} for more details. Combining the gains of integrability for $0\le m \in L^1$ in Lemma \ref{Gain-of-Integrability} with the classical result  \cite[Lemma 3.1 ]{CafCraKocSwi96}, we enhance the regularity for $u$. This is the content of our next result:

\begin{theorem}\label{Gains-of-Regularity}
Let $(u, m)$  be a weak solution to \eqref{eq_main}. Suppose A\ref{A1}-A\ref{A2} hold and $\tilde{f}$ is a bounded function in $B_1$, then 
	\[	
	u \in  W^{2, \frac{d(p-1)}{d-1}}_{loc}(B_{1});
	\]
	and, we have the estimate
	\[
	\| u\|_{W^{2, \frac{d(p-1)}{d-1}}(B_{1/2})} \leq C\left(\|u\|_{L^{\infty}(B_1)} +  \|m\|_{L^1(B_1)}^{\frac{1}{p-1}} \right),
	\]
	where $C>0$ is a constant. 
\end{theorem}

%%%%%%%%%%%%%%%%%%%%%%%%%%%%%%%%%%%%%%%%%%%%
%%%%%%%%%%%%%%%%%%%%%%%%%%%%%%%%%%%%%%%%%%%%
%%%%%%%%%%%%%%%%%%%%%%%%%%%%%%%%%%%%%%%%%%%%

\begin{proof}
We combine classic regularity results for Sobolev spaces with Lemma \ref{Gain-of-Integrability} to prove the statements. Because of $W^{2,p}$-regularity theory and the hypothesis on the operator $F$, the result follows if the right-hand side of the first equation in \eqref{eq_main} belongs to $L^{\frac{d(p-1)}{d-1}}(B_{1/2})$.
 
It follows from Lemma \ref{Gain-of-Integrability} that $m \in L^{\frac{d}{d-1}}(B_{1/2})$ with the estimate \eqref{est:gain-of-integrability}. Then, there exist constants $C_1>0$ and $C_2>0$ such that
\[ 
	\int_{B_{1/2}} |m^{\frac{1}{p-1}}|^{\frac{d(p-1)}{d-1}} \leq C_1 \|m\|_{L^1(B_{3/4})} + C_2 \| \tilde{f}\|_{L^{\infty}(B_1)}.
\] 
The desired conclusion follows since $m\in L^1(B_1)$ and $\tilde{f}$ is a bounded function. 

As a consequence, we have the existence of a constant $C>0$ so that the estimate for $u$ holds. 
\end{proof}

%%%%%%%%%%%%%%%%%%%%%%%%%%%%%%%%%%%%%%%%%%%%
%%%%%%%%%%%%%%%%%%%%%%%%%%%%%%%%%%%%%%%%%%%%
%%%%%%%%%%%%%%%%%%%%%%%%%%%%%%%%%%%%%%%%%%%%
%%

\begin{corollary}
Let $(u, m)$  be a weak solution to \eqref{eq_main}. Assume $p > d$ and A\ref{A1}-A\ref{A2} hold. Then $u \in \mathcal{C}_{loc}^{1, \alpha}(B_1)$, with the estimate
\[
\|u\|_{\mathcal{C}^{1, \alpha}(B_{1/2})} \leq \left(\|u\|_{L^{\infty}(B_1)} +  \|m\|_{L^1(B_1)}^{\frac{1}{p-1}} \right),
\]
where $C$ is a positive constant and $\alpha$ is given by
\[
\alpha := 1 - \dfrac{d-1}{p-1}.
\]
\end{corollary}

%%%%%%%%%%%%%%%%%%%%%%%%%%%%%%%%%%%%%%%%%%%%
%%%%%%%%%%%%%%%%%%%%%%%%%%%%%%%%%%%%%%%%%%
%%%%%%%%%%%%%%%%%%%%%%%%%%%%%%%%%%%%%%%%%%%%

\begin{proof}
The corollary follows from Theorem \ref{Gains-of-Regularity} with general Sobolev inequalities in Sobolev spaces \cite{Evansbook}.
\end{proof}

\bigskip

\noindent{\bf Acknowledgements} The authors are grateful to H\'ector Chang-Lara, Edgard Pimentel, Boyan Sirakov and Hugo Tavares for their interest, suggestions, and insightful comments on this manuscript. We would also like to thank the anonymous referees for their helpful comments and suggestions, which helped us improve the paper. PA is partially supported by the Portuguese government through FCT-Fundação para a Ciência e a Tecnologia, I.P.,
under the projects UID/MAT/04459/2020 and UIDP/00208/2020. JC is partially supported by FAPERJ-Brazil (\#E26/202.075/2020) and CNPq-Brazil under Grant No. 408169/2023-0.

\bigskip

	\bibliographystyle{plain}
	\bibliography{biblio}
	
	\bigskip

	\vspace{1cm}

\noindent\textsc{P\^edra D. S. Andrade (Corresponding Author)}\\
\noindent\small{ Department of Mathematics, Paris Lodron Universit\"at Salzburg, Hellbrunnerstrasse 34, A-5020 Salzburg, Austria.\\
\noindent\texttt{pedra.andrade@plus.ac.at}}

\bigskip

\noindent\textsc{Julio C. Correa}\\
\noindent\small{Instituto de Matem\'atica e Estat\'istica, Universidade do Estado do Rio de Janeiro, 20550-013, Maracan\~a, Rio de Janeiro - RJ, Brazil.\\\noindent\texttt{julio.correa@ime.uerj.br}.}

\bigskip

\end{document}